\documentclass[16pt]{article}
\usepackage{amsmath}
\usepackage{amssymb}
\usepackage{amsfonts}
\usepackage{mathrsfs}
\usepackage{mathbbol}
\usepackage{eurosym}
\usepackage{color}
\usepackage{url}
\usepackage{graphicx}
\usepackage{wrapfig}

\makeatletter \g@addto@macro\bfseries{\boldmath} \makeatother \oddsidemargin=5mm
\newtheorem{theorem}{Theorem}
\newtheorem{acknowledgement}[theorem]{Acknowledgements}

\newtheorem{corollary}[theorem]{Corollary}

\newtheorem{definition}[theorem]{Definition}
\newtheorem{example}[theorem]{Example}
\newtheorem{exercise}[theorem]{Exercise}
\newtheorem{lemma}[theorem]{Lemma}

\newtheorem{proposition}[theorem]{Proposition}

\newenvironment{proof} {\par\noindent{\bf Proof.}} {\hfill$\scriptstyle\square$}

\def\c{\mathbb{C}}
\def\s{\mathbb{S}}
\def\q{\mathbb{Q}}
\def\z{\mathbb{Z}}
\def\r{\mathbb{R}}

\def\overK{\overline{K}}

\def\scre{\mathscr{E}}
\def\scri{\mathscr{I}}
\def\scrh{\mathscr{H}}
\def\scrot{\mathscr{O}}

\def\onr{\mathrm{O}_{n}(\r)}
\def\glnK{\mathrm{GL}_{n}(K)}
\def\glnr{\mathrm{GL}_{n}(\r)}
\def\glnR{\mathrm{GL}_{n}(\r)}
\def\gltK{\mathrm{GL}_{2}(K)}
\def\gltr{\mathrm{GL}_{2}(\r)}
\def\glnk{\mathrm{GL}_{n}(k)}
\def\glnK{\mathrm{GL}_{n}(K)}
\def\sltr{\mathrm{SL}_{2}(\r)}
\def\psltr{\mathrm{PSL}_{2}(\r)}
\def\sltK{\mathrm{SL}_{2}(K)}
\def\sltra{\mathfrak{sl}_{2}(\r)}
\def\slnr{\mathrm{SL}_{n}(\r)}

\def\slnK{\mathrm{SL}_{n}(K)}
\def\sonK{\mathrm{SO}_{n}(K)}
\def\sonr{\mathrm{SO}_{n}(\r)}
\def\sonR{\mathrm{SO}_{n}(\r)}
\def\sonr{\mathrm{SO}_{n}(\r)}
\def\sotr{\mathrm{SO}_{2}(\r)}
\def\autt{\mathrm{Aut}(T)}

\def\onK{\mathrm{O}_{n}(K)}
\def\gr{\mathrm{Gr}}

\def\ac{algebraically closed~}
\def\cc{conjugation complete}
\def\iff{if and only if~}
\def\tr{\mathrm{Tr}}
\def\sp{\mathrm{Sp}}

\def\ig{$\mathcal{IG}$}
\def\tig{$\mathcal{TIG}$}

\begin{document}
\title{A few remarks on invariable generation in
infinite groups}

\author{Gil Goffer \\
The Weizmann Institute of Science\\
 76100, Rehovot, Israel \and
Gennady A.Noskov \\
 The Sobolev Institute of Mathematics \\
644043, Omsk, Russia  }

\date{February 12, 2018}

\maketitle

\begin{abstract}
  A subset $S$ of a group $G$ invariably generates $G$ if $G$ is generated by $\{ s^{g(s)} | s\in S\} $
  for any choice of $g(s)\in G, s\in S$.
  A topological group $G$ is said to be \ig~ if it is invariably generated by some subset $S\subseteq G$, and \tig~ if it is topologically invariably generated by some subset $S\subseteq G$.
  In this paper we study the problem of (topological) invariable generation for linear groups and
  for automorphism groups of trees.
       Our main results show that the Lie group $\sltr $ and the automorphism group of a
       regular tree are \tig,
  and that the groups $PSL_m(K) ,m\geq 2$ are not \ig~ for countable fields
  of infinite transcendence degree over a prime field.
\end{abstract}

\tableofcontents

\section{Introduction}


The notion of invariable generation of groups goes back at least as far as Jordan's paper  of 1872.

\begin{theorem}[\cite{jordan,serre-jordan}]. Assume that a finite group $G$ acts transitively on a set
$X$ of cardinality at least 2. Then there exists $g\in G$ displacing every element of $X$, i.e.
 $g$ acts fixed-point-freely on $X$.
\end{theorem}
  Independently, but more than hundred years later, J. Wiegold  introduced the class
 $\mathfrak{X}$ of groups $G$ with the
 property that whenever $G$ acts transitively on a set $X$ of cardinality at least 2,
 there exists an  element $g\in G$ acting fixed-point-freely on $X$. (See \cite{wiegold76}).
  Thus Jordan's theorem above says that every finite group belongs to the class $\mathfrak{X}$.
  Not only Wiegold failed to mention Jordan,
  but also he claims indirectly that Jordan's theorem is obvious.
 ``Observe that, in particular, finite groups and soluble groups are in $\mathfrak{X}$  -- obvious facts once noted, though I know of no reference to
them in the literature''. (See \cite{wiegold76}).
 Wiegold gave two more characterizations of the class $\mathfrak{X}$.
 Namely, $G\in \mathfrak{X}$ \iff  for every proper subgroup $H \leq G$ the union
 $ \bigcup \{H^g : g \in G\}$ does not cover the whole $G$.
  (Here, we use the common notation $H^g$ to denote the conjugate $g^{-1}Hg$ of $H$.)
 Secondly,  $G\in \mathfrak{X}$ \iff every  subset,
 containing at least one representative
 of every conjugacy class of $G$,  generates $G$.


 These characterizations lead naturally to the following definitions (of our own,
 though the  concepts implicitly occurred in above conversation).
 A subset $S$ of a group $G$ is called \textsf{\cc~}(in $G$) if it meets every non-trivial
 conjugacy class of $G$. A proper  subgroup $W$ of a group $G$ is called a \textsf{Wiegold subgroup} if $G$ is covered by conjugates of $W$, i.e. $G=\bigcup _{g\in G}W^{g}$. This is equivalent to saying that $W\leq G$ is proper, and $W$ is \cc.

   In these terms it is easy to see that the class $\mathfrak{X}$ consists of those groups that
  have no Wiegold subgroups, or equivalently, groups in which every \cc~ set generates $G$.
  Now Jordan's theorem becomes really ``obvious'' 
  for finite groups, by a counting argument.
  Indeed, arguing by contradiction, suppose that $G$ is a finite group with  a Wiegold subgroup $W$.
  Then $G=\bigcup W^g$. There are at most $|G:W|$ members in this union, each member is of cardinality $|W|$ and all members contain the identity
  element, hence the cardinality of the union is strictly less than $|G:W|\cdot |W|=|G|$,
  contradicting that $W$ is Wiegold.

 Independently of Wiegold, but somewhat later, J. D. Dixon has
 invented the notion of invariable generation for finite groups (see \cite{dixon}).
 A subset $S$ of a finite group $G$ \textsf{invariably generates}
 $G$ if $G$ is generated by $\{ s^{g(s)} | s\in S\} $
  for any choice of $g(s)\in G, s\in S$.
    This notion was motivated by Chebotarev's Density Theorem \cite{chebotarev-dense}.
 W. M. Kantor, A. Lubotzky and A. Shalev,  following Dixon, have defined a
 (not necessarily finite) group $G$ to be
 \textsf{invariably generated} if any \cc~ subset of $G$ generates $G$ (see \cite{KLS2}).
 They also discovered that the class \ig~
 of invariably generated groups coincides with  Wiegold's class $\mathfrak{X}$.
 Since then a few works have been done showing that specific families of groups are invariably
 generated by given sets, or proving other groups are not \ig~ by any subset.
 See the papers \cite{dixon,gelander-convergence,gelander-congruence,gelander-thompson,KLS1,KLS2,wiegold76,wiegold77}
 containing
 a lot of information on the class \ig.

 In this paper we study the notion  of \textsf{topological invariable generation}
 (abbreviated to \tig).
  \tig~ is a weaker property that \ig, and there are plenty of examples for groups which are \tig~ but not \ig.
  For instance, it is proved in this paper that the Lie group $\sltr $ is \tig, whereas
  it seems extremely difficult to determine the status of  $\sltr $
  endowed with the discrete topology.

  In this paper we study the problem of topological invariable generation for linear groups and
  for automorphism groups of trees.
       Our main results show that the Lie group $\sltr $ and the group of tree automorphisms
       $\mathrm{Aut}(T)$ are \tig~ (sections 6 and 7),
  and that the groups $PSL_m(K) ,m\geq 2$ are not \ig~ for countable fields
  of infinite  transcendence degree over the prime field (section 8).
  In section 2 we give an alternative description of \tig~  in terms of Wiegold subgroups and
  in terms of fixed-point free actions.
  We clarify the relationship between the notions of Wiegold, Borel and parabolic subgroups
  in sections 3-5.

\section{Preliminary results}

We shall formulate now the invariable generation problem in a topological setup.
 An action $G\curvearrowright X$ of  a  topological group $G$ on a topological
 space $X$ is \textsf{continuous} if the map $(g,x)\mapsto gx$ from
 $G\times X$ to $X$ is continuous. For a subgroup $W\leq G$ and an element $g\in G$, let $W^{g}=g^{-1}Wg$ denote the conjugate of $W$ by $g$.
  A subset $S$ of $G$ is called \textsf{\cc~} (in $G$) if it meets every non-trivial
 conjugacy class of $G$.
 We call a proper closed subgroup $W$  of $G$ a \textsf{Wiegold subgroup} if
   $G$ is covered by conjugates of $W$, i.e. $G=\bigcup _{g\in G}W^{g}$, or equivalently if $W\leq G$ is proper, and $W$ is \cc.

 An action $G\overset{\alpha}{\curvearrowright} X$ of a group $G$ on a set
 $X$ has \textsf{fixed-point property}
 (for short, $\alpha$ is an \textsf{(FP)-action})
  if every $g\in G$ has a fixed point in $X$.
  Alternatively, $\alpha$ is an FP-action \iff $G$ is covered by the point stabilizers,
  i.e. $G=\bigcup_{x\in X} G_x$, where the \textsf{stabilizer } of $x$ is
  $G_x=\{g\in G: gx=x\}$.

\begin{lemma}\label{l-4-equivalences}
The following conditions on a topological group $G$ are equivalent:
 \begin{description}
   \item[I.]
  ~ $G$ contains a Wiegold subgroup.
   \item[II.]
 ~ There is a continuous transitive FP-action $G\curvearrowright X$ on a Hausdorff topological
 space $X$ of cardinality at least 2.
   \item[III.]
  There exists a \cc~ subset $S$ of $G$ that does not generate   $G$
  topologically.
 \end{description}
 \end{lemma}

\begin{proof}
\begin{description}
  \item[I $\Rightarrow$ II.] Suppose that $G$ contains a Wiegold subgroup, say $W$.
  The natural action $G\curvearrowright (G/W)$ is transitive, and since $W$ is proper,
  the cardinality of $G/W$ is at least 2.
  Moreover, every $g\in G$ is contained in some conjugate $ W^{h^{-1}}, h\in G$.
  The inclusion $g\in W^{h^{-1}}$
  is equivalent to the equality $ghW= hW$,  which in turn
  means that $g$  fixes the point $hW$ in the  topological quotient
  $G/W$. Thus the action $G\curvearrowright (G/W)$ satisfies
  the fixed-point property.

  \item[II $\Rightarrow$ I.] Suppose that II holds for the action $G\curvearrowright X$.
  We claim that the stabilizer  $G_x$ is Wiegold for every $x\in X$.
  Fix $x\in X$.
  First, $G_{x}$ is proper since the action is transitive.
  Next,
  the stabilizer $G_{x}$   is a closed subgroup of $G$.
  This is because, as the action is continuous, the orbit map
  $\pi:G \rightarrow X: g \mapsto gx$ is continuous,
  and $G_{x} = \pi^{-1}(\{x\})$, with $\{x\}$ closed.
   By fixed-point property, $G$ is covered by stabilizers,
   i.e. $G=\bigcup_{y\in X} G_y$.
   And by transitivity every $y\in X$ is of the  form $y=gx$
   for some $g\in G$, thus $G$ is covered by stabilizers of the form $G_{gx}=gG_{x}g^{-1}$ for $g\in G$, and so $G_x$ is Wiegold.

  \item[I $\Rightarrow$ III.]
   Let $W$ be a Wiegold subgroup of $G$. Then the subset $W$ of $G$ is \cc, and since $W$ is a proper closed subgroup of $G$, it does not generate $G$
   topologically.

 \item[III $\Rightarrow$ I.]
 Suppose, there exists a \cc~ subset $S$ of $G$ that does not generate   $G$
  topologically.
  The closure $W=\overline{\langle S\rangle}$ is a proper closed subgroup of $G$. Since $S$ was \cc, so is $W$.
\end{description}
 \end{proof}

We say that the topological group $G$ is \textsf{topologically invariably generated}
 (abbreviated to \tig) if the following equivalent conditions (i)-(iii)  are satisfied.
 Note, that these conditions, being the negations of I-III, are all equivalent

 \begin{description}
   \item[(i).]
  ~ $G$ contains no Wiegold subgroup.
   \item[(ii).]
 ~ If $G\curvearrowright X$ is continuous transitive action on a topological
 space $X$ of cardinality at least 2 then $G$ contains a  fixed-point-free element.
   \item[(iii).]
  Every \cc~ subset of $G$ generates   $G$
  topologically.
 \end{description}

 In case $G$, being equipped with the discrete topology, is \tig~ we say that $G$ is \textsf{invariably generated} (abbreviated to \ig~). For discrete groups, conditions I-III were introduced in \cite{wiegold76} and for finite groups condition III was introduced in \cite{dixon}.
   Clearly, if $G$ is \ig~ then $G$ is \tig~ relative to every group topology on $G$.
 The converse is not true -- for instance the free group $F_2$ of rank 2
  is  \tig~ relative to the profinite topology but it is  not \ig~ (see \cite{wiegold77}).

\begin{example}
 \emph{If $G$ is an abelian group then the conjugacy classes are just 1-element subsets of $G$. So
  the \cc~ sets are those containing $G-\{1\}$. Thus every conjugation complete
  set generates $G$. Hence $G$ has no Wiegold subgroup and $G$ is \ig.}
\end{example}

\begin{example}\label{example-single}
 \emph{ Consider the extreme case when $G$ contains precisely two conjugacy classes:
  $\{1\}$ and $G-\{1\}$.
  Let us show that $G$ is \ig~ \iff $G\simeq\z/2$.
  Clearly, $G=\z/2$ is \ig~. Conversely, suppose that $G$ is \ig.
  Every 1-element set $\{g\}\subset (G-\{1\})$ is conjugation complete, hence $G$ is
  generated by $g$ and therefore $G$ is cyclic.
  It is easily verified that the only cyclic group that has precisely two conjugacy classes is $\z/2$. Thus $G=\z/2$.
 It is proved by D. Osin that any countable torsion free group can be embedded into a finitely generated group with exactly two
conjugacy classes (see \cite{osin}). }
\end{example}
 These examples suggest, that intuitively
 the group $G$ is \ig~ if it contains ``many'' conjugacy  classes.

 \begin{lemma}[\cite{wiegold76, wiegold77}]
 \begin{description}
   \item[i]  The class \ig~ is closed under extensions.
   \item[ii] The class \ig~ is closed under restricted direct products with arbitrarily many factors.
   \item[iii] The class \ig~ is not closed under unrestricted direct products and not closed
  under passage to subgroups.
   \item[iv] The class \ig~ contains all finite groups and all solvable groups.
\end{description}
 \end{lemma}

In \cite{wiegold77} an \ig~-group is  presented, whose commutator subgroup  is outside \ig~.
 In the present paper it is shown that $\sltr$ is \tig, whereas it contains a
 discrete free group $F_2$, which is not \ig.

 Slightly modifying Wiegold's arguments  we obtain the following properties of
 the class \tig.

 \begin{lemma}\label{profinite-solvable}
 The class \tig~ contains all profinite groups and all topological solvable groups.
 \end{lemma}

\bigskip

 Not so many classes of groups are known to be \tig :

\begin{enumerate}

\item Virtually solvable groups (including finite groups) are \tig~ (See \cite{wiegold76}).

\item
 The pro-finite completions of all arithmetic groups with the Congruence Subgroup Property
 are \tig~  (See \cite{KLS2}).
\end{enumerate}

This list is extended in the present paper by the group $\sltr$ and the group of automorphisms of a regular tree, which are shown to be in \tig.

\bigskip

 The list of known non-\ig~ groups is lengthy but rather chaotic:

\begin{enumerate}
\item Groups with one single non-trivial conjugacy class, except $\z/2$.

\item
The finite index net subgroups in  groups $\mathrm{SO}(Q,\z)$ where $Q$ is a rational quadratic
form of signature $(2, 2)$  (see \cite{gelander-congruence}).

 \item
 All non-abelian connected compact Lie groups
 (see \cite{gelander-convergence}).


\item $\mathrm{SL}_{n}(\mathbb{R})$ for $n>2$  (see \cite{KLS2}).

\item Any non-virtually-solvable linear algebraic group over an
algebraically closed field, for instance
 $\mathrm{SL}_{n}(\mathbb{C})$ for $n\geq 2$ (see \cite{KLS2}).

\item Non-elementary convergence groups (including Gromov hyperbolic groups)
           \cite{gelander-convergence}.

\item The Thompson groups $T$ and $V$ ( see \cite{gelander-thompson}).
%

\end{enumerate}

This list is extended in the present paper by the groups $PSL_m(K),m\geq 2$ where $K$ is a countable field of infinite transcendence degree over a prime field.


\section{Parabolics and stabilizers as Wiegold subgroups}
Let $\glnK$ be a general linear group of degree $n\geq 2$ over the
 (topological) field $K$.
 Consider the standard left action $\glnK\curvearrowright K^n$ of $\glnK$ on $K^n$.
 Let  $\gr(n,r)$ denote the  \textsf{Grassmannian} of all  $r$-dimensional linear subspaces of
 $K^n$.
  In the next proposition we give various conditions on the field and the action of a group $G$ on
 the Grassmannian that provide that certain stabilizers are Wiegold subgroups.

\begin{proposition}\label{p-stabilizer-is-wiegold}
  Let $G$   be a  subgroup of $\glnK$ and let $X\subseteq \gr(n,r),~n\geq r>1 $
  be a $G$-invariant subset with more than one element.
   \begin{enumerate}

     \item Suppose that   the action $G \curvearrowright X$ is  transitive and satisfies
   the fixed-point property FP.
    Then every setwise stabilizer $G_P ~(P\in X)$ is a Wiegold subgroup of $G$.
%

    \item In case $K$ is \ac we can drop the fixed-point assumption at the expense of assuming that
 $X=\gr(n,r)$.
 Namely, let $K$ be an \ac field and suppose a subgroup $G\leq\glnK$  acts transitively
 on  $\gr(n,r)$. Then for every $r$-dimensional linear subspace
 $L$ of $K^n$ the (setwise) stabilizer $G_L$ is Wiegold.
 In particular $\glnK$ and $\slnK$ are not \ig.

\item
 Suppose  a subgroup $G\leq\glnr$ acts transitively on the Grassmannian of
 planes $\gr(n,2) ~(n\geq 3)$. Then for every plane  $L$ of $\r^n$ the stabilizer $G_L$ is Wiegold.
 In particular $\glnr$, $\slnr$, $\onr$, $\sonr$ are not \ig.
   \end{enumerate}

\end{proposition}

\begin{proof}
 The first assertion is just a direct consequence of lemma \ref{l-4-equivalences} in terms of action
 of $G$ on Grassmannian. The second assertion follows from the first since
 every
  $g\in \glnK$ fixes  some $r$-dimensional linear subspace $L$ of $K^n$.
 (This is true for $r=1$ since $K$ is \ac and for arbitrary $r$ an easy induction
 does the job).

 As far as the third assertion is concerned, it is well known, via complexification,
 that every $g\in \glnr$ fixes a subspace of dimension 1 or 2 (see \cite[p.~279]{axler}, \cite[Prop.12.16]{kostrikin-manin}).
 Moreover, to each nonreal eigenvalue corresponds an invariant plane.
 It follows that if $n\geq 2$ then every $g\in \glnr  $ fixes
 a subspace of dimension 2.
  Then for every plane  $L$ of $\r^n$ the stabilizer $G_L$ is Wiegold.
 Thus $\glnr$ is not \ig~  and moreover, every subgroup $G\leq\glnr$ which acts transitively
 on $\gr(n,2)$ turns out to be non-\ig.
  \end{proof}

\textbf{Symplectic groups.}
 The \textsf{symplectic group} $Sp_{2n}(k)$ is  defined as the set of all linear
transformations of a   $2n$-dimensional vector space $k^{2n}$ over the field $k$ that preserve a
non-degenerate, skew-symmetric, $k$-bilinear form $\omega$ on $k^{2n}$.
    A subspace $L$ of a symplectic vector space is called \textsf{Lagrangian} if $L=L^\perp$,
    where $L^\perp$ is an orthogonal complement of $L$ relative to $\omega.$
    A subspace $L$ is Lagrangian \iff $\dim L=n$ and $L$ is \textsf{isotropic}, i.e. $\omega| L$
    is identically zero.
    The set of Lagrangian subspaces of symplectic space
    $(k^{2n},\omega)$ is called the \textsf{Lagrangian Grassmannian}
     and is denoted by $\mathcal{L}(k^{2n})$.
  The group $Sp_{2n}(k)$ acts transitively on $\mathcal{L}(k^{2n})$.
  Indeed  by Witt's theorem (1937)
  every linear isomorphism $f:L\rightarrow L'$ of Lagrangian subspaces
   can be extended to a symplectic automorphism of $k^{2n}$.
   Moreover, the following is true:
   \begin{theorem}[\cite{duistermaat}, Theorem 3.4.3]
    Let $(E, \omega)$ be a symplectic vector space over a field $k$ and let
    $A:(E, \omega) \rightarrow (E, \omega)$
   be a symplectic mapping   such that all its eigenvalues are
   contained in $k$.  Then there exists an $A-$invariant Lagrangian subspace $L$ of $(E, \omega)$.
      Moreover, if $k$ is \ac then every symplectic automorphism of $(E, \omega)$ fixes  (setwise)
      some    Lagrangian subspace.
   \end{theorem}

  It follows that $Sp_{2n}(k)_L$ is Wiegold for every Lagrangian subspace $L$.

  \textbf{Wiegold subgroups over real closed fields.} These fields have the same
  first-order properties as $\r$.
 A field $K$ is said to be \textsf{real} if  $-1$ is not a sum of squares in $K$
 (see \cite[Corollary 9.3, Chapter VI]{lang}).
 A field $K$ said to be \textsf{real closed} if it is real, and if any algebraic extension of
 $K$ which is real must be equal to $K$.
 By a \textsf{real closure} we shall mean a real closed field $L$ which is algebraic over $K$.
 Every  real field admits  a real closure.
 If $K$ is real closed then every polynomial of odd degree in $K[X]$ has a root in $K$ and
$\overK:=K(\sqrt{-1})$ is the algebraic closure of $K$.


\begin{proposition}\label{l-real-closed}
 Let $K$ be a real closed field.
 Suppose  a subgroup $G\leq\glnK$ acts transitively on Grassmannian of
 planes $\gr(n,2)~(n\geq 3)$.
 Then for every plane  $L$ of $K^n$ the stabilizer $G_L$ is Wiegold.
 In particular $\glnK,\slnK,\onK,\sonK$ are not \ig.
  \end{proposition}

\begin{proof}
  We assert that every $g\in\glnK) ~ (n\geq 2)$ has invariant plane in the column space $(K(\sqrt{-1}))^n$.
 If all eigenvalues of $g$ are real, i.e. belong to $K$, then $g$ is triangulizable over $K$
 and so has an invariant plane. Now suppose that there is a non-real eigenvalue
 $a+b\sqrt{-1}\in\overK, (a,b\in K) b\neq 0.$
 Let $w=u+\sqrt{-1}v \in\overK^n $ be an eigenvector  corresponding to $g$.
 Then
 \begin{equation}\label{eq-gw}
 gw=gu+\sqrt{-1}gv=(a+b\sqrt{-1})(u+\sqrt{-1}v)=(au-bv)+\sqrt{-1}(av+bu)
 \end{equation}
 from which it follows that the space $Ku+Kv$ is $g$-invariant.
 In fact this space is 2-dimensional.
 Indeed, suppose $u=cv, 0\neq c\in K$, then substituting $u$ to equation \ref{eq-gw}
 we obtain that $c^2=-1$ contradicting to assumption that $K$ is real.
 To complete the proof we apply now proposition \ref{p-stabilizer-is-wiegold} and obtain that the
 action $G\curvearrowright\gr(n,2)$ satisfies $FP$ and is transitive by assumption
 so $G_L$ is Wiegold.
\end{proof}

 In attempt to generalize proposition 9  to the subspaces of arbitrary dimension $r$
 we came to

 \begin{lemma}\label{l-r-invariant-from-fields}
Let $K$ be an \ac field of degree $d\geq 2$ over its subfield $k$ and suppose that
 the extension $K|k$ is separable.
 Then every $g\in \glnk$ fixes a linear subspace in $k^n$ of dimension smaller than or equal   $d$.
  \end{lemma}

 \begin{proof}
 By Primitive Element Theorem there exists a \textsf{primitive} element $\zeta\in K$, i.e.
 $K=k(\zeta)$.
 It follows from finiteness assumption that $\zeta$ is algebraic over $k$.
 Then $K=k(\zeta) = k[\zeta]$, and $k(\zeta)$ is finite over $k$
 \cite[Proposition 1.4., Chapter V]{lang}.
 Moreover, the degree $d=[k(\zeta) : k]$ is equal to the degree of the irreducible
 polynomial $p(X)$ of $\zeta$ over $k$
 (loc.cit).
 The proof of \cite[Proposition 1.4., Chapter V]{lang} shows that
 the powers $1,\zeta,\ldots,\zeta^{d-1}$ form a basis for the space $K$ over $k$.
 Let $V=K\otimes_k k^n$ be an extension of scalars from $k$ to $K$. This is an
 $n$-dimensional space over $K$.
 The action of $g\in \glnk$ on $k^n$ extends naturally to an action on $K\otimes_k k^n$.
 Since $K$ is \ac there is an eigenvector $v\in V$ with eigenvalue $\lambda\in K$.
 Write uniquely $\lambda=\sum_i l_i \zeta^i, v=\sum_j \zeta^j\otimes v_j$, where
 $l_i\in k, v_i\in k^n$.
 We assert that $g$ fixes the space $\sum_i kv_i$ (whose dimension is clearly at most $r$).
 Since $gv=\lambda v$ we have
 $$
 gv =\sum_j \zeta^j\otimes gv_j=\lambda v=
   (\sum_i l_i\zeta^i)(\sum_j\zeta^j\otimes v_j)=
 \sum_{i,j} l_i\zeta^{i+j}\otimes v_j.
 $$
 Rewriting $\zeta^{i+j}$ as a $k$-linear combinations of $1,\zeta,\ldots,\zeta^{r-1}$
 we obtain that
 $$\sum_{i,j} l_i\zeta^{i+j}\otimes v_j=\sum_k \zeta^k \otimes v_k'$$
 where $v_k'$ are $k$-linear combinations of $v_j$.
 It follows from above that $gv_j=v_j'$, and the lemma follows.
 \end{proof}

 Unfortunately, the applicability of the above lemma is very much restricted by
 the Artin-Schreier theorem which asserts that if a field
 $K$ is algebraically closed with $k$ a subfield such that  $1 < [K: k] < \infty$
 then $K = k(i)$ where $i^2 = -1$, and $k$ has characteristic 0 \cite[Corollary 9.3, Chapter VI]{lang}.

\section{Borel versus Wiegold}

 The following is a slight variation on the result in \cite{KLS2}.

\begin{theorem}
Let $G$ be a linear algebraic group over an algebraically closed field $k$
and let $B$ be its Borel
subgroup (=maximal connected closed solvable subgroup).
 1) If $G$ is not virtually solvable then the normalizer $N_G(B)$ is a  Wiegold subgroup of $G$.
 2) If furthermore $G$ is connected then $B$ is itself Wiegold.

\end{theorem}

\begin{proof}
1) Firstly let us prove that $N_G(B)$ is a proper subgroup of $G$.
 Suppose to the contrary that $N_G(B)=G$, i.e. $B$ is normal in $G$.
 Note that, being connected, $B$ is contained in the connected component
 $G_0$ of $G$. Since $B$ is normal in $G$, it also normal in $G_0$.
 Note that $B$ is proper in $G_0$ since otherwise $G$ would be
 a finite extension of $B=G_0$, i.e. virtually solvable.
  Finally by Chevalley result \cite[theorem 11.16]{borel-lga} $B$
  is self-normalizing in $G_0$, contradicting the normality of $B$.
  Hence $N_G(B)$ is a proper subgroup of $G$.

 By a theorem of Steinberg (see Theorem 7.2 of \cite{steinberg-endomorphisms}),
 every automorphism of $G$
 fixes some Borel subgroup of $G$.
 This implies that if $g$ is any element of $G$, then $g^{-1}B'g=B'$ for some Borel subgroup $B'$
 of $G$.
 Thus the union of the normalizers $N_G(B')$ over the Borel subgroups $B'$ of $G$ equals $G$.
 Since the Borel subgroups are all conjugate, it follows that $\bigcup_g (N_G(B))^g = G$, thus
 $N_G(B)$ is Wiegold.

2) In this case we make use of Grothendieck's Covering Theorem \cite[Theorem 4.11]{conrad-reductive}:
 Let $G$ be a connected linear algebraic group over \ac field $k$ and let
 $B \leq G$ be a Borel subgroup then $$G(k) =\bigcup_{g\in G} g^{-1}B(k)g\space .$$

 This immediately implies that $B$ is Wiegold.

\end{proof}

 \textbf{2-dimensional case}.
 The \textsf{standard Borel subgroup} $B$ of $\gltK$  is the group of all invertible
 upper triangle  matrices
 $\left(
\begin{smallmatrix}
a & b \\
0 & d
\end{smallmatrix} \right)$.
 For the Borel subgroup of $\sltK$ we require in
addition that $ad = 1$.
 The following exercise provides a supply of Wiegold subgroups in the
 2-dimensional case, cf. \cite[Exercise 23,p.547]{lang}.

\begin{exercise} Let $K$ be a field in which every quadratic polynomial has a root.
 Let $B$ be the Borel subgroup of $\gltK$ or of $\sltK$.
 Show that $B$ is Wiegold, and conclude that $\gltK$ is not \ig.
\end{exercise}

\textbf{Solution}. We have to show that every matrix $A=\left(
\begin{smallmatrix}
a & b \\
c & d
\end{smallmatrix} \right)\in \gltK$
 is conjugate to a matrix in $B$.
  If $b= 0$  we  conjugate $A$ by
 $\left(
\begin{smallmatrix}
0 & 1 \\
-1 & 0
\end{smallmatrix} \right)$.
 If $b\neq 0$ the conjugation $A^X $ by the matrix
 $X=$
$\left(
\begin{smallmatrix}
1 & 0 \\
x & 1
\end{smallmatrix} \right)$
 gives
 $(A^X)_{2,1}=bx^2+(d-a)x-c$ which can be made  zero by assumption.
 We conclude that $B$ is a Wiegold subgroup (and so $\gltK$ is not \ig~).

 An example of countable $K$ satisfying the condition of the above exercise
 is the "quadratic closure" of $\q$ in $\c$.
  How many such fields are there?

\section{Compact Lie groups and their generalizations}

As we mentioned above every  nonabelian connected compact Lie group is not \ig~.
 Here
 we generalize this remark to definably compact groups
 according to \cite{berarducci}.

We consider groups definable in an o-minimal expansion $\mathcal{M} = (R,<,+,\ldots )$ of a real
closed field $(R,<,+,\ldots)$.
 Classical examples of such groups are the (real) algebraic subgroups
of the general linear group $\glnR$, for instance the orthogonal group $\sonR$.
 Less classical examples of
definable groups can be found in \cite{peterzil-steinhorn} or in \cite{str}.

 \begin{theorem}
If $G$ is a nonabelian definably connected definably compact  group in an o-minimal expansion
 $\mathcal{R}$ of
a real closed field $R$, then for any maximal definable abelian subgroup $T$ of $G$, $G$ is the
union of the conjugates of $T$, i.e. $T$ is a Wiegold subgroup in $G$.
 In particular, $G$ is not \ig.
\end{theorem}

This immediately follows from the main result of \cite{berarducci}.


\section{$\sltr$ is \tig~}

 \begin{theorem}\label{thm-sltr-tig}
 The Lie group $\sltr$ with its canonical topology is topologically invariably generated.
 \end{theorem}
  Structure of the proof:   The theorem will immediately follow from the key lemma
  \ref{l-closed-cc-is-sltr} in which  we prove that if a closed subgroup $G$ of $\sltr$ is conjugation complete then $G=\sltr$. To prove lemma \ref{l-closed-cc-is-sltr} we show that every proper closed subgroup of $\sltr$ has eigen-spectrum which is strictly contained in the eigen-spectrum of $\sltr$. As eigen-spectrum is conjugacy invariant, a \cc~ subgroup must have the same eigen-spectrum of $\sltr$, and so the lemma follows.
   The chapter is ordered as follows: We calculate the eigen-spectrum of $\sltr$ in section \ref{s-conj-sltr}. Then we describe the 1-parameter subgroups of $\sltr$ in section \ref{s-1-parameter}.
   In section \ref{s-1dim-spectra} we describe spectra of 1-dimensional closed subgroups of $\sltr$ and show that their spectra are (much) smaller than that of $\sltr$. In section \ref{s-2dim-spectra} we do the same in case of 2-dimensional subgroups. In section \ref{s-main} we prove the main lemma and theorem \ref{thm-sltr-tig}.

\begin{corollary}
  The Lie group $\psltr$ with its canonical topology is topologically invariably generated.
\end{corollary}

\begin{proof}
  Suppose that $\psltr$ admits a closed Wiegold subgroup $H$. Then the pre-image of $H$ under the quotient map $p\colon \sltr \to \psltr$ is a closed Wiegold subgroup of $\sltr$, contradicting Theorem \ref{thm-sltr-tig}.
  \end{proof}

 \subsection{Conjugacy classes of $\sltr$}\label{s-conj-sltr}
 Following \cite{conrad-decomposing} we remark that conjugacy in $\sltr$
 should not be confused with conjugacy in the
larger group $\gltr$.

 \begin{theorem}[\cite{conrad-decomposing}]
  Let $g\in\sltr$.
\begin{enumerate}
\item
 If $(\mathrm{Tr}g)^2>4$ then $g$ is conjugate to a unique matrix of the form
 \[
 \begin{pmatrix}
  \lambda & 0 \\
  0 & 1/\lambda\\
\end{pmatrix}
 \]
 with $\lambda \in \mathbb{R}$, $|\lambda|>1$.

  \item
  If $(\mathrm{Tr}g)^2=4$ then $g$ is conjugate to exactly one of

\begin{equation}\label{}
 \pm I_2,
  \pm
\begin{pmatrix}
  1 & 1 \\
  0 & 1\\
\end{pmatrix},
 \pm
\begin{pmatrix}
  -1 & 1 \\
  0 & -1\\
\end{pmatrix}.
\end{equation}

\item
 If $(\mathrm{Tr}g)^2<4$ then $g$ is conjugate to a unique  matrix of the form
 \[
  \begin{pmatrix}
  \cos\theta & -\sin\theta \\
  \sin\theta & \cos\theta \\
\end{pmatrix}
 \] other than $\pm I_2.$

\end{enumerate}
\end{theorem}

  For a set of (complex) matrices $M$ let $\sp(M)$ to denote the
  \textsf{eigenvalue spectrum} of $M$, i.e. the set of all eigenvalues
 of matrices in $M$ (counting without multiplicity).
  The following is an immediate consequence of the above result.

\begin{corollary}
  $\sp\left( \sltr\right) =\r^\times \cup \mathbb{S}$, where $\r^\times=\mathbb{R}%
-\left\{ 0\right\} $ \ and \
 $\s=
 \left\{ z\in \mathbb{C}:\left\vert z\right\vert =1\right\}$.

\end{corollary}

\subsection{One-parameter subgroups of $\sltr$}\label{s-1-parameter}
 Recall that a \textsf{one-parameter subgroup} in $\glnr$ is a homomorphism $\r\rightarrow\glnr$
 and it is necessarily of the form $f(t) = e^{At}$ where $A$ is the matrix $f'(0)$.
 We describe here the conjugacy classes of one-parameter subgroups in $\sltr$.
 For that we need to know the orbits of $\sltr$-action on its Lie algebra.
 We set
 $$\sltra  = \{X\in M_2(\r)~|~ \tr X = 0\}=
 \{(
 \begin{smallmatrix}
a & b \\
c & -a%
\end{smallmatrix}
)
 :a,b,c\in\r
 \}
 $$
 to denote the Lie algebra of the Lie group $\sltr$.
 For a nonzero $2\times 2$ real matrix
 $X=(
 \begin{smallmatrix}
a & b \\
c & -a%
\end{smallmatrix}%
 )\in \sltra$ we  have that the characteristic polynomial is $x^2-(a^2+bc)$ and therefore its eigenvalues are $\pm \sqrt{a^2+bc}$.
 There are three cases:

  1)  \emph{$a^2+bc>0$.}
In this case both eigenvalues are real and non-zero. Denote them by $\pm t$. Then $X$ is conjugate by some $g\in\gltr$ to the diagonal matrix
 $
( \begin{smallmatrix}
t & 0 \\
0 & -t%
\end{smallmatrix}.%
 )$
 To see that $X$ is also conjugate to
 $(\begin{smallmatrix}
t & 0 \\
0 & -t%
\end{smallmatrix})$
 by the smaller group $\sltr$  we replace $g$ by
  $\left(\begin{smallmatrix}
\det(g)^{-1} & 0 \\
0 & 1%
\end{smallmatrix}\right)
 g $.%

Matrices of the form
 $(\begin{smallmatrix}
t & 0 \\
0 & -t%
\end{smallmatrix})$
constitute a 1-dimensional Lie subalgebra
$\mathfrak{a}=
\r (\begin{smallmatrix}
1 & 0 \\
0 & -1%
\end{smallmatrix})
$
  of $\sltra$ called the \textsf{split Cartan subalgebra}.
  Exponentiating  elements of $\mathfrak{a}$ we obtain
  the \textsf{split Cartan subgroup}
 $$A=\left\{
 \begin{pmatrix}
e^{t} & 0 \\
0 & e^{-t}%
\end{pmatrix}
:t\in \r\right\},$$
  which is a 1-parameter subgroup of $\sltr$ corresponding to
  subalgebra $\mathfrak{a}$.

  2)
 \emph{$a^2+bc=0$}.
In this case $X$ has zero eigenvalue of multiplicity 2.
  Then $X$ is conjugate by  some $g\in\gltr$ to a nilpotent matrix of the form
 $
 (\begin{smallmatrix}
0 & d \\
0 & 0%
\end{smallmatrix}), 0\neq d\in\r.%
 $
 Again we can make $g$ to be in $\sltr$ replacing  $g$ by
 the matrix
$\left(
  \begin{smallmatrix}
    \det(g)^{-1} & 0 \\
    0 & 1 \\
  \end{smallmatrix}
\right)$$g$.

Matrices of the form
 $(\begin{smallmatrix}
0 & t \\
0 & 0%
\end{smallmatrix})$
constitute a 1-dimensional Lie subalgebra
$\mathfrak{n}=
\r (\begin{smallmatrix}
0 & 1 \\
0 & 0%
\end{smallmatrix})$
  of $\sltra$ called \textsf{nilpotent subalgebra} $\mathfrak{n}$.
  Exponentiating elements of $\mathfrak{n}$ we obtain
  the \textsf{unipotent subgroup}
 \[
 U=
 \begin{pmatrix}
1 & \r \\
0 & 1%
\end{pmatrix}
\]
  which is a 1-parameter subgroup of $\sltr$ corresponding to the subalgebra $\mathfrak{n}$.

 3) \emph{$a^2+bc<0$.}
In this case $X$ has nonzero purely imaginary conjugate eigenvalues $\pm\theta i ~~(\theta=\sqrt{-a^2-bc})$.
 We claim that $X$ is conjugate to one of the matrices
$
\pm\theta\left(
  \begin{smallmatrix}
    0 & 1 \\
    -1  & 0 \\
  \end{smallmatrix}
\right)
$
 by a matrix from $\sltr$.
 There are real linearly independent column vectors $u,v$ such that
 $X(u+iv)=\theta i(u+iv)=\theta (-v+iu)=Xu+iXv$, from what we obtain
 $Xu=-\theta v, Xv=\theta u$.
  It follows that a 2-by-2 matrix $(u,v)\in M_2(\r)$ with columns $u,v$
  is invertible.
  We have
 $X(u,v)=(Xu,Xv)=\theta (-v,u)=\theta   (u,v)
 \left(
   \begin{smallmatrix}
   0 & 1 \\
   -1 & 0 \\
   \end{smallmatrix}
 \right).
 $

It follows that
\begin{equation} (u,v)^{-1}X(u,v)=\theta
   \begin{pmatrix}
  0 & 1 \\
 -1 & 0 \\
   \end{pmatrix}
 .
\end{equation}
\label{eq-conjug-x-by-uv}

 We wish to modify $g=(u,v)$ to be in $\sltr$.
 By multiplying the matrix $g$ by $\sqrt{|g|}^{-1}$, we may assume $\det(g)=\pm 1$.
 If $\det(g)=-1$ then we replace $g$ by
$g \left(
   \begin{smallmatrix}
   1 & 0 \\
   0 & -1 \\
   \end{smallmatrix}
 \right),
 $
which is in $\sltr$, and we get that $g^{-1}Xg=-\theta$
$g \left(
   \begin{smallmatrix}
   0 & 1 \\
   -1 & 0 \\
   \end{smallmatrix}
 \right).
 $

Matrices of the form
 $(\begin{smallmatrix}
0 & t \\
-t & 0%
\end{smallmatrix})$
constitute a 1-dimensional Lie subalgebra
$\mathfrak{k}=
\r (\begin{smallmatrix}
0 & 1 \\
-1 & 0%
\end{smallmatrix})$
  of $\sltra$ called the \textsf{non-split Cartan subalgebra}.
  Exponentiating elements of $\mathfrak{k}$ we obtain
   a 1-parameter subgroup $K=\sotr$ of rotations called
  the \textsf{non-split Cartan subgroup}:

$$ K=\left\{
 \begin{pmatrix}
\cos t & \sin t \\
-\sin t & \cos t%
\end{pmatrix}
:t\in \r\right\}.$$

Clearly $\sotr$ is a compact subgroup of $\sltr$. In fact, it is a maximal closed subgroup of
$\sltr$ \cite[Ch.VI, section 1]{glasner}.

\subsection{One-dimensional Lie subgroups of $\sltr$ and their spectra}\label{s-1dim-spectra}

\begin{lemma}\label{l-1dim-small-spectra}
 The eigen-spectrum of every closed 1-dimensional subgroup in $\sltr$ is
 a proper subset in $\sp(\sltr)$.
\end{lemma}
 \begin{proof}

 Suppose that $G$ is a Lie subgroup of $\sltr$ such that
 its component $G^0$ is 1-dimensional, so is a 1-parameter subgroup of one of the above types
 $A,U,K$.

  1)  In case $G^{0}=A$ we have
\begin{equation*}
A=G^{0}\trianglelefteq G\trianglelefteq N(A),
\end{equation*}

where $N\left( A\right) $ is the normalizer of $A$ in $\sltr$.
 By \textsf{Bruhat decomposition} $N(A)=A\cup wA$, where
$w=%
(\begin{smallmatrix}
0 & 1 \\
-1 & 0%
\end{smallmatrix})$
 is the \textsf{Weyl element} (see \cite[p.539]{lang}).
 We conclude that
 $$\sp\left( G\right) \leq
 \sp \left( N\left( A\right) \right) =
 \sp(A)\cup\sp(wA).$$
 Note that (surprisingly but luckily) $\sp(wA)$ equals
 \[
\sp \left\{
 \begin{pmatrix}
0 & e^{-t} \\
-e^t & 0%
\end{pmatrix},
t\in\r%
\right\} =
 \{\pm i\}
 \]
 and indeed $\sp(G)\leq \r^{\times}\cup \{\pm   i\}$ is a proper subset of $\sp(\sltr)$.

2) If $G^{0}=U$ is unipotent, then it is easy to see $N(U)=B$, where
 $B=\left\{ \pm
(\begin{smallmatrix}
e^{t} & s \\
0 & e^{-t}%
\end{smallmatrix})%
;t,s\in R\right\} .$
 Hence
 $G^{0}\trianglelefteq G\trianglelefteq N(U)=B$  and clearly $\sp\left(
G\right) \leq \sp\left( B\right) =\mathbb{R}^{\times}$.

 3) Finally, in case $G^{0}=K$ we have $N\left( K\right) =K$, see \cite[Ch.VI, sec. 1]{glasner},
  hence
 $\sp\left( G\right) \leq \sp\left( K\right) =\mathbb{S}$.

 \end{proof}
\subsection{2-dimensional Lie subgroups of $\sltr$ and their spectra}\label{s-2dim-spectra}
 A two-dimensional Lie algebra over $\r$ is either abelian
 or solvable. Furthermore a solvable Lie algebra  $\mathfrak{s}_2$
 has generators $e,f$ and defining relation $[ef]=e$ ( see \cite{jacobson}).
 Using the above classification of 1-dimensional subalgebras it is easy
 to conclude that there are no two-dimensional abelian subalgebras in $\sltra$.
 On the other hand there is only one (up to conjugacy) subalgebra isomorphic to $\mathfrak{s}_2$,
 namely
 the Borel subalgebra
\[
\mathfrak{b}=
 \left\{
\begin{pmatrix}
{t} & s \\
0 & {-t}%
\end{pmatrix}%
;t,s\in \r
 \right\}.
\]

 The (connected) component $B^0$ of the Borel subgroup
 $$
B=\left\{ \pm
\begin{pmatrix}
e^{t} & s \\
0 & e^{-t}%
\end{pmatrix}%
;t,s\in R\right\} .$$
 has $\mathfrak{b}$ as its Lie algebra.
  It follows that
 $B^0$ is the only (up to conjugacy) 2-dimensional connected subgroup of $\sltr$.
 We conclude with

\begin{lemma}\label{l-2dim-small-spectra}
 Let $G$ be a Lie subgroup of $\sltr$ such that
 its component $G^0$ is a component $B^0$ of the Borel subgroup $B$.
Then $G\leq B$ and
$\sp\left( G\right) =\r^{\times}$.
\end{lemma}


\subsection{Main lemma and theorem \ref{thm-sltr-tig} } \label{s-main}

\begin{lemma}\label{l-closed-cc-is-sltr}
 If a closed subgroup $G$ of $\sltr$ is
conjugation complete then $G=\sltr$.
\end{lemma}

\begin{proof}
Suppose, on the contrary, that there exists a proper closed  and conjugation complete subgroup
 $G$ in  $\sltr$.
 It follows from conjugation completeness that
 $\sp (G) =\sp( \sltr)$.

 By Cartan's theorem any closed subgroup of $\glnr$ is a Lie group.
 (A short elegant proof can be found in \cite{adams}, pages 17-19).
 Hence $G$ is a Lie subgroup of $\sltr$.
 The connected component $G^{0}$ of $G$ is a proper closed normal Lie subgroup of $G$.
  Since $G$ is a proper
 subgroup of 3-dimensional connected Lie group $\sltr$, we conclude that $\dim G^{0}=0,1$ or 2.

  If $\dim G^{0}=0$ then $G$ is at most countable discrete subgroup of
$\sltr$ and thus its spectrum is at most countable, so $\sp\left( G\right) $ is properly contained
in $\r^\times \cup \s = \sp \left( \sltr\right) $. Contradiction.

  If $\dim G^{0}=1$ then $G^0$ is a 1-parameter subgroup of $\sltr$ and by lemma
    \ref{l-1dim-small-spectra} the eigen-spectrum of $G$ is
 a proper subset in $\sp(\sltr)$. Contradiction.

 Finally if $\dim G^{0}=2$ then $G^{0}$ is the connected component of the Borel subgroup
 and by lemma
    \ref{l-2dim-small-spectra} the eigen-spectrum of $G$ is
 a proper subset in $\sp(\sltr)$. Contradiction.

   The lemma and thereby theorem \ref{thm-sltr-tig} are proved.
\end{proof}

\section{$\autt$  is \tig~} \label{autt is tig}

 A \textsf{(simplicial) graph} $\Gamma=(V,E)$ consists of  a set of \textsf{vertices} $V$ and a family
 $E$ of \textsf{edges} which are 2-element subsets of $V$.
  A \textsf{(simple) path} from vertex $x$ to vertex $y$ is a sequence of vertices
  $\pi: x=x_0,x_1,\ldots,\ldots, x_n=y$ such that every
  subset $\{x_i,x_{i+1}\}$ is an edge and every two subsequent edges are distinct.
  The number $n\geq 0$ is the \textsf{length} of $\pi$.
  Define the \textsf{distance} $d(x,y)$  between vertices $x$ and $y$ to be the minimum
  of the lengths of simple paths joining them.
  A \textsf{loop} is a path of positive length such that $x_0=x_n$.
  A \textsf{tree} is a (simplicial) graph without loops.
 The \textsf{valence} or \textsf{degree} of a vertex $x$ in a tree $T$ is the number of edges that
 contain $x$.
 A \textsf{regular} $d$-tree is a tree with every vertex having degree $d$.
   Consider the group $\autt$ of \textsf{automorphisms (=isometries)} of $T$,
equipped with the standard pointwise convergence topology.
  Then $\autt$ turns into a totally disconnected, locally compact, unimodular topological group.
 In this section we prove the following theorem:

\begin{theorem}\label{thm-autt-tig}
  The automorphism group $\autt$ of a $d$-regular tree is \tig~ for every $d\geq 2$.
\end{theorem}

\subsection{The classification of isometries} \label{classification AutT}
  According to \cite{tits} for a single automorphism $g$ of a $d$-regular tree $T$
   there are  three mutually exclusive cases to consider:
 \begin{itemize}
   \item $g$ is \textsf{elliptic}, i.e. $g$ has a fixed point in $V(T)$,
   \item $g$ is an \textsf{inversion}, i.e. there is an edge which is
flipped by $g$,
   \item $g$ is \textsf{hyperbolic} with an \textsf{axis} $A_g$ : that  is to say $A_g$
    is a subtree with all vertices of degree 2 and $g$ acts on  $A_g$ by translation
 by a positive amount, denoted $|g|$. The number $|g|$ is called the \textsf{translation length}
 of $g$.
 \end{itemize}
 We will denote by $\scre,\scri,\scrh$ the classes of elliptics, inversions and
 hyperbolics, respectively. It is easy to see that they are all nonempty and
 pairwise disjoint. Moreover each of these classes is invariant under conjugation.

 Observe that conjugation translates the axes and the fixed points, i.e.
 if $g,h\in \autt$ and $h$ is hyperbolic  then $A_{g^{-1}hg}=gA_h$.
 If $h$ is elliptic fixing $v\in V(T)$, then $g^{-1}hg$ fixes
 $gv$, and if $h$ is an inversion flipping the edge $e\in E(T)$, then $g^{-1}hg$ flips $ge$.

\subsection{Vertex stabilizers}
 Let $T$ be a $d$-regular tree, $d\geq 2$. We denote the pointwise stabilizer in $\autt$ of a subset $M\subset V(T)$ by $Stab_{\autt}(M)$, or by $Stab(M)$.
 When $v\in V(T)$, we may simply write $Stab(v)$ to mean $Stab(\{v\})$.
 When $K\leq \autt$ is a subgroup, we denote by $Stab_{K}(M)$ the stabilizer of $M$ inside $K$.
 For a vertex $v\in V(T)$ and $n\geq 0$ we denote by $B(v,n)\subset V(T)$ the closed ball of radius $n$  about $v$.

  Fix a vertex $v\in V(T)$ and denote by $(T,v)$ the rooted tree $T$ with a root $v$. We identify the stabilizer subgroup $Stab(v)\leq \autt$ with the group $Aut(T,v)$ of automorphisms of the rooted tree $(T,v)$.
The group $Aut(T,v)$ is a profinite group. Indeed, it is the inverse limit of its finite quotients $Aut(T,v)/Stab(B(v,n))$.
Just like finite groups, every profinite group is \tig . The proof is due to \cite{KLS2}, we sketch it here for completeness:

\begin{lemma}
  A profinite group is \tig .
\end{lemma}

\begin{proof}(sketched).
  Let $G$ be a profinite group.
Every proper closed subgroup of a profinite group is contained in a maximal open subgroup
$M$. Since $M$ has finite index it cannot be conjugation complete, hence G contains no Wiegold subgroups.
\end{proof}

\begin{corollary} \label{AutTv is tig}
  The group $Aut(T,v)$ of automorphisms of a rooted $d$-regular tree, $d\geq 2$, is \tig.
\end{corollary}

In the next section we present the conjugacy classes of $\autt$ and observe the relations between conjugacy in $\autt$ and conjugacy in $Aut(T,v)$.

\subsection{Conjugacy classes of $\autt$} \label{AutT conjugacy classes}

 We follow \cite{gns} for the description of the conjugacy classes of the groups $\autt$ and $Aut(T,v)$.
 For $g\in \autt$ we define the \textsf{orbital type} of $g$ to be the labeled graph $T_g$ of
 $g^{\z}$-orbits, where $g^{\z}$ is the cyclic group generated by $g$.
  That is, the vertices of $T_g$ are $g^{\z}$-orbits; An edge  $(o_1,o_2)\in E(T_g)$ is connecting two orbits $o_1,o_2\in V(T_g)$ if the orbits contain vertices $x_1\in o_1,x_2\in o_2$ which were connected by an edge $(x_1,x_2)\in E(T)$ in the original tree; Each $g^{\z}$-orbit is labeled by its cardinality, in $\mathbb{N}\cup\{\infty\}$.
 Two orbital types are called equivalent if they are isomorphic as labeled trees.
 We call an orbital type \textsf{elliptic} if it contains an orbit labeled by $1$. Observe that $g$ has an elliptic orbital type \iff $g$ is elliptic.

Any conjugate of an elliptic element (inversion, hyperbolic) is clearly elliptic (resp. inversion,
hyperbolic) as well. Therefore, conjugacy classes of $\autt$ can be described separately in each of
these cases.

\begin{lemma} \label{conj in AutT}
 If $g,g'\in \autt$ are both elliptic or both inversions then they
 are conjugate to each other \iff their orbital types are equivalent.
 If  $h,h'\in \autt$ are hyperbolic then they conjugate to each other \iff they have equal translation
length.
\end{lemma}

In the rooted tree $(T,v)$, all elements are elliptic. We have therefore the following characterization:

 \begin{lemma}\label{conj in Aut(T,v)}
 Elements $g,g'\in Aut(T,v)$ are conjugate to each other (in $Aut(T,v)$) \iff their orbital types are equivalent.
 \end{lemma}

For the full proofs of lemma \ref{conj in AutT} and lemma \ref{conj in Aut(T,v)} see \cite{gns}.

Denote by $\scrot$ the collection of all elliptic orbital types of elements in $\autt$. A direct corollary is the following:

\begin{corollary}
The collection of all orbital types of elements in $Aut(T,v)$ is $\scrot$.
\end{corollary}

\begin{proof}
  Clearly  the collection of all orbital types of elements in $Aut(T,v)$ is contained in $\scrot$, as $Aut(T,v)=Stab_\autt(v)$ is a subgroup of $\autt$. For the other inclusion observe that every element $g$ in $\autt$ which has an elliptic orbital type is elliptic, therefore stabilizes some $u\in V(T)$. Let $h\in \autt$ be an element such that $h.u=v$, then $hgh^{-1}$ is in $Stab(v)$ and has the same orbital type as $g$.
\end{proof}


\subsection{Generating a vertex transitive subgroup}

Let $s\in \autt$ be an elliptic element stabilizing a vertex $v\in V(T)$ and assume that the action
of $s^\z$ is transitive on all spheres about $v$. We call such an element $v$-\textsf{spherically
transitive}.
 An elliptic element $s\in \autt$ is called \textsf{spherically
transitive} if it is $v$-spherically transitive for some $v$.
 For any $v\in V(T)$ there exists a $v$-spherically transitive element in $\autt$ (see \cite{bass-renorm}).
  If $s$ is
 $v$-spherically transitive, then the orbital type of $s$ is an infinite ray
$[v=x_{0},x_{1},x_{2},\dots)$ with labels $f(x_{0})=1$, $f(x_{n})=d(d-1)^{n-1}$, $n\geq1$ (that is,
the size of the $n^{th}$-sphere about $v$). Hence, all spherically transitive elements in $\autt$
are conjugate.


\begin{lemma}
\label{hyperbolic and spherically transitive implies transitivity} Let $h$ be a hyperbolic
 automorphism of translation length $|h|=1$ and let $s$ be a spherically transitive element.
 Then $\langle h,s\rangle$ acts transitively on the vertices of $T$.
\end{lemma}

\begin{proof}
  Let $v$ denote the fixed point of $s$, $A_h$ denote the axis along which $h$ translates and $m$ denote the distance between $v$ and $A_h$ (possibly $m=0$). Since $h$ acts transitively on the vertices of $A_h$, it is enough to show that any vertex in $T$ can be mapped by $\langle h,s\rangle$ into $A_h$.
  Indeed, let $y\in V(T)$ be a vertex. For large enough $k$ we have that $d(h^k.y,v)>m$. Since every sphere of radius at least $m$ around $v$ intersects $A_h$, we have that some element in $\langle s\rangle$ maps $h^k.y$ into $A_h$, and so the lemma is proved.
  \end{proof}

The following is a direct corollary from the last lemma and the structure of conjugacy classes in $\autt$.
\begin{corollary} \label{cc set is vertex transitive}

  every conjugation complete subset of $\autt$ generates a vertex transitive subgroup.
\end{corollary}

\begin{proof}
  Such a set must contain hyperbolic elements of all translation lengths $l\in \{1,2,\dots\}$ and elliptic elements of all orbital types $i\in \scrot$. In particular it must contain a hyperbolic element of translation length $1$ and a spherically transitive element, therefore it satisfies the assumptions of lemma \ref{hyperbolic and spherically transitive implies transitivity}.
\end{proof}

\subsection{Proof of theorem \ref{thm-autt-tig}}
We prove the following rephrasement of theorem \ref{thm-autt-tig}.
\begin{theorem}
 Every closed \cc~ subgroup $H$ of the automorphism group $\autt$ of an $m$-regular tree
 $T$,  $(m\geq 2)$,
 coincides with $\autt$.
\end{theorem}

\begin{proof}
 Let $H\leq \autt$ be a closed \cc~ subgroup.
 By corollary \ref{cc set is vertex transitive} $H$ is vertex-transitive.
 Furthermore, since $H$ is \cc, it contains elements of all elliptic orbital types. Let $\{r_i\}_{i\in\scrot}\subseteq H$ be a set of representatives in $H$ for all orbital types $i\in \scrot$.
Each $r_i$ is elliptic, and has a non-empty set of fixed points in $V(T)$. For every $i\in \scrot$ let $v_i$ be a fixed point of $r_i$.
Fix $v\in V(T)$ and denote by $K=Stab_{\autt}(v)$.
Since $H$ is vertex transitive, we can find elements $h_i\in H$ s.t $h_i.v_i=v$ for all $i$. Consider the collection $\mathcal{C}=\{h_ir_ih_i^{-1}\}_{i\in \scrot}\subseteq H$. As conjugation does not affect the orbital type, this collection contains elements of all orbital types in $\scrot$. Further, all elements in $\mathcal{C}$ stabilize $v$. So $\mathcal{C}$ is a \cc~ collection in $K$, with respect to conjugacy inside $K$.
By corollary \ref{AutTv is tig} $K=Aut(T,v)$ is \tig , so in particular, the collection $\mathcal{C}\subseteq H$ generates $K$, which implies that $K\leq H$.

 We obtain that $H$ is vertex  transitive and contains the vertex-stabilizer group
 $K=Stab_{\autt}(v)$.
 It follows that $H=\autt$.
\end{proof}

\section{Uncountably many countable non-\ig~-groups}

\subsection{Laws in groups}

A non-empty word $w=w\left( x_{1},\ldots ,x_{n}\right) $ (= string of
symbols in alphabet $x_{1}^{\pm 1},\ldots ,x_{n}^{\pm 1}$) is \textsf{%
reduced} if \ it does not contain adjacent symbols of the form $%
x_{i}^{\varepsilon }x_{i}^{-\varepsilon },\varepsilon =\pm 1$. A non-empty word $w=w\left(
x_{1},\ldots ,x_{n}\right) $ is \textsf{reduced on the tuple of elements} $\left( g_{1},\ldots
,g_{n}\right) $ of a group $G$ if $w$ is reduced and every word of the form $x_{i}^{m}$\ does not
occur as a subword in $w$ when $\left\vert m\right\vert \geq \left\vert g_{i}\right\vert $, where
 $|g_i|$ denotes the order of $g_i$
( we
of course assume that $\left\vert g_{i}\right\vert =\infty $ when $g_{i}^{%
\mathbb{Z}}$ is infinite cyclic).
 A tuple of elements $\left( g_{1},\ldots ,g_{n}\right) $ of a
 group $G$ is \textsf{free} if the group $\langle g_{1},\ldots ,g_{n}\rangle $ generated by them
 is a free product of cyclics $g_{i}^{\mathbb{Z}}$, i.e. the group
 $\langle g_{1},\ldots ,g_{n}\rangle $, generated by $g_{1},\ldots ,g_{n}$,
 is isomorphic to the free product $%
g_{1}^{\mathbb{Z}}\ast \cdots \ast g_{n}^{\mathbb{Z}}$ under the natural homomorphism from
$g_{1}^{\mathbb{Z}}\ast \cdots \ast g_{n}^{\mathbb{Z}}$ to
$\langle g_{1},\ldots ,g_{n}\rangle $.
 A word $w$ is said to be a \textsf{law}
 in a group $G$ if $w$ becomes trivial whatever values the arguments $x_{i}$
are assigned from $G$; i.e. $w\left( g_{1},\ldots ,g_{n}\right) =1$ for all $%
g_{i}\in G$.

The proof of the following lemma is routine and is left to the reader.

\begin{lemma}
A tuple $\left( g_{1},\ldots ,g_{n}\right) $ over a group $G$ is free \iff~ every non-empty word
$w\left( x_{1}^{\pm 1},\ldots ,x_{n}^{\pm 1}\right) $, which is reduced over the tuple $\left(
g_{1},\ldots ,g_{n}\right) $, does not vanish on \ $\left( g_{1},\ldots ,g_{n}\right) $, i.e.
$w\left( g_{1},\ldots ,g_{n}\right) \neq e$. In other words a tuple $\left( g_{1},\ldots
,g_{n}\right) $ is free \iff there are no non-trivial relations between the elements
$g_{1},\ldots ,g_{n}$.
\end{lemma}

More generally, consider a free product $G[x]=G\ast x^{\mathbb{Z}}$ of $G$ and the infinite cyclic
group $x^{\mathbb{Z}}$. An element $w\in G[x]$ will be called a \textsf{monomial in variable} $x$\
\textsf{over the group} $G$. In case $w\not\in G$ we say that $w$ is \textsf{non-constant}. A
monomial $w$ induces a map $w:G\rightarrow G$, where $w(g)$ is the image of $w$ under the natural
group homomorphism
\begin{equation*}
G[x]\rightarrow G,
\end{equation*}%
taking $x$ to $g$ and fixing $G$ pointwise. The subset $V(w)=\{(g{\in
G:w(g)=1\}}$ is called the \textsf{principal algebraic set }of $G$\textsf{\ }%
corresponding to $w$. Every monomial $w\in G[x]$ can be written as a word $%
w=a_{0}x^{l_{1}}a_{1}x^{l_{2}}\cdots a_{m-1}x^{l_{m}}a_{m}$ with $a_{i}$-s in $G$. We say that this
word expression is \textsf{reduced over }$G$ if it does not contain a subword $x_{{}}^{\pm
1}ax_{{}}^{\mp 1}$ with $a$ in the center of $G$. If $a$ is central then replacing $x_{{}}^{\pm
1}ax_{{}}^{\mp 1}$ by $a$ does not change the principal algebraic set $V\left( w\right) $.\ After
finitely many such replacements any word $w$ can be brought to a reduced form (which, of course,
could occur to be empty or constant). A (non-empty) reduced non-constant word $w\left( x\right)
=a_{0}x^{l_{1}}a_{1}x^{l_{2}}\cdots a_{m-1}x^{l_{m}}a_{m}$ over $G$ is said to be a
\textsf{generalized law} in a group $G$ if $w$ becomes trivial whatever values the arguments $x$
are assigned from $G$; i.e. $w\left( g\right) =1$ for all $g\in G$.

The following is the simplified version of the result by Mikhalev and
Golubchik, showing that there is no generalized law in $\mathrm{GL}%
_{m}\left( K\right) ,\left( m\geq 2\right) $ in case of infinite field $K$.

\begin{theorem}[\protect\cite{gm}]
\label{thm-MG-1-variable}Let $K$ be an infinite field and $w\left( x\right)
=g_{0}x^{i_{0}}g_{1}x^{i_{1}}\cdots g_{m}x^{i_{m}}g_{m+1}$ be a monomial
over the group $\mathrm{GL}_{m}\left( K\right) ,\left( m\geq 2\right) $. If $%
g_{1},\ldots ,g_{m}$ are all non-central in $\mathrm{GL}_{m}\left( K\right) $ then the principal
algebraic set $V\left( w\right) =\left\{ g\in
GL_{m}\left( K\right) :w\left( g\right) =1\right\} $ is a proper subset of $%
GL_{m}\left( K\right) $.
\end{theorem}

The same result (with the same proof) is true for  the groups $\mathrm{SL}%
_{m}(K)$ and $\mathrm{PSL}_{m}(K)$.\bigskip

\subsection{Rational points}

As in \cite{borel-free} we identify implicitly an algebraic group $G$ with
the group $G\left( \Omega \right) $ of its points in a "universal domain" $%
\Omega $, i.e. an algebraically closed extension of infinite transcendence degree over a prime
field (recall that the \textsf{prime field }is either
the rational number field $\mathbb{Q}$, or a finite field of prime order $%
\mathbb{F}_{p}$).

Assume now that $G$ \ is defined over a field $K$ of infinite transcendence degree over a prime
field $k$. We shall need the following lemma which guaranties that, under some natural conditions,
given a family of proper algebraic subsets \ $V_{i}$ $\left( i\in \mathbb{N}\right) $ of $G$, the
set of $K$-rational points \ outside of a countable union \ $\bigcup V_{i}$ $\left( i\in
\mathbb{N}\right) $ is nonempty. (This is kind of an algebraic analog of Baire category theorem).

\begin{lemma}[\protect\cite{borel-free}, \S 2,Lemma 2]
\label{borel-lemma}Let $K$ be a field of infinite transcendence degree over a prime field. Let $X$
\ be an irreducible unirational $K$-variety. Let $L$
be a finitely generated subfield of $K$ containing a field of definition of $%
X$, and $V_{i}$ $\left( i\in \mathbb{N}\right) $ a family of proper algebraic subsets of $X$
defined over an algebraic closure $\overline{L}$ of $L$. Then $X\left( K\right) $ is not contained
in the union of the \ $V_{i}$ $\left( i\in \mathbb{N}\right) $.
\end{lemma}

\subsection{Building up free tuples}

\begin{lemma}[main]
\label{main-lemma}Let $K$ be a field of infinite transcendence degree over a prime field $k$. $\
$Let $C_{1},\ldots ,C_{n}$ be non-trivial distinct conjugacy classes of the group
$G=\mathrm{PSL}_{m}\left( K\right)$,  where $n,m\geq 2$. Suppose that $c_{1}\in C_{1},\ldots
,c_{n-1}\in C_{n-1}$ are such that $\left( c_{1},\ldots ,c_{n-1}\right) $ is a free tuple. Then
there is a representative $c_{n}\in C_{n}$ such that the tuple $\left( c_{1},\ldots ,c_{n}\right) $
is free.
\end{lemma}

\begin{proof}
\ \ \ \ \ Suppose  that $c_{1},\ldots ,c_{n-1}$ are given and we have to
construct $c_{n}$. Fix some $c\in C_{n}$. We seek $g\in \mathrm{PSL}%
_{m}\left( K\right) $ such that the $n$-tuple $\left( c_{1},\ldots ,c_{n}\right) $ is free with
$c_{n}=g^{-1}cg$. We will first find out $g$ in a larger group $G\left( \Omega \right)
=\mathrm{PSL}_{m}\left( \Omega \right) $ and then apply Borel's lemma. Let $W$ be the set of all
non-empty
words $w\left( x_{1},\ldots ,x_{n}\right) $ which are reduced on the tuple $%
\left( c_{1},\ldots ,c_{n-1},c\right) $. Thus $W=\{w=w\left( x_{1},\ldots
,x_{n}\right) :$ $w$ is reduced,$\ x_{i}^{m}$ does not occur in $w$ when $%
\left\vert m\right\vert \geq \left\vert c_{i}\right\vert $ and $x_{n}^{m}$ does not occur in $w$
when $\left\vert m\right\vert \geq \left\vert c\right\vert $ $\}$.

Let $w=w\left( x_{1},\ldots ,x_{n}\right) \in W$. The replacement $%
x_{1}\mapsto c_{1},\ldots ,x_{n-1}\mapsto c_{n-1},x_{n}\mapsto x^{-1}cx$ gives rise to the monomial
$w\left( c_{1},\ldots ,c_{n-1},x^{-1}cx\right) $ over $G$ in variable $x$, which is not necessarily
reduced over $G$. Indeed, for every maximal occurrence of $x_{n}^{m}$ in $w$ the result of
substitution $x_{n}\mapsto x^{-1}cx$ contains a monomial $\left( x^{-1}cx\right) ^{m}$ which is
clearly not reduced. The replacement $\left( x^{-1}cx\right) ^{m}$ by $x^{-1}c^{m}x$ is a non-empty
monomial $w\left( c_{1},\ldots ,c_{n-1},x^{-1}cx\right) _{red}$ with variable $x$ over $G$ and this
monomial is reduced over $G$. The set
\begin{eqnarray}
G(w) &=&\left\{ g\in G\left( \Omega \right) :w\left( c_{1},\ldots
,c_{n-1},g^{-1}cg\right) =1\right\}  \\
&=&\left\{ g\in G\left( \Omega \right) :w\left( c_{1},\ldots ,c_{n-1},g^{-1}cg\right)
_{red}=1\right\}
\end{eqnarray}

is an algebraic subset of $G\left( \Omega \right) $ given by $w\left( c_{1},\ldots
,c_{n-1},x^{-1}cx\right) _{red}=1$. It is defined over the field $L=k\left( c_{1},\ldots
,c_{n-1},c\right) $ which is finitely generated over the prime subfield $k$ and contains the field
of definition $k $ of $G$. The irreducible components $G\left( w\right) _{1},\ldots ,G\left(
w\right) _{n_{w}}$of $G\left( w\right) $ are defined over the separable closure
 $L_{s}$ \cite[\S 12.3]{borel-free}.
 Each $G\left( w\right) $ is a proper subset of $G\left( \Omega
\right) $ since otherwise $w\left( c_{1},\ldots
,c_{n-1},x^{-1}cx\right) _{red}=1$ is a non-empty generalized law on $G=%
\mathrm{PSL}_{m}\left( \Omega \right) $, which contradicts Theorem \ref%
{thm-MG-1-variable}. The family of algebraic subsets $\left\{ G\left( w\right) _{i}:w\in W\right\}
$ is countable. Finally, the group $G$, being reductive, is unirational over $k$ \cite[section
18.2]{borel-lga}. All above
means that we can apply Borel's lemma \ref{borel-lemma} with $X=G$, $%
L=k\left( c_{1},\ldots ,c_{n-1},c\right) $ and $\left( V_{i}\right) $ the family of all irreducible
components of all $G\left( w\right) ,w\in W$.

We conclude that $G\left( K\right) $ is not contained in the union $\bigcup _{w\in W}G\left( w\right)
$.   The complement
\begin{equation}
G\left( \Omega \right) -\underset{w\in W}{\bigcup }G(w)
\end{equation}%
consists of $g\in G\left( \Omega \right) $ such that the $n$-tuple $\left( c_{1},\ldots
,c_{n-1},g^{-1}cg\right) $ is free.$\ $Hence $G\left( K\right) $ contains a point $g$ outside of
union $\bigcup _{w\in W}G\left( w\right) $, so the corresponding tuple $\left( c_{1},\ldots
,c_{n-1},g^{-1}cg\right) $ is free and $K$-rational.
\end{proof}

\subsection{When is $\mathrm{PSL}_{m}\left( K\right) $ an \ig $?$}

\begin{theorem}
If a field $K$ is a countable and has infinite \ transcendence degree over the prime field, then
the group\ $\mathrm{PSL}_{m}\left( K\right) $ is not \ig~ for $m\geq 2$.
\end{theorem}

\begin{proof}
The group $G=\mathrm{PSL}_{m}\left( K\right) $ is countable. Then $%
G-\{e\}=\cup C_{i},\left( i\in \mathbb{N}\right) $ , where $C_{i}$ are all
non-trivial distinct conjugacy classes. Inductively applying Lemma \ref%
{main-lemma}, we find out representatives $c_{i}$ so that the subgroup $%
\langle c_{i}:i\in \mathbb{N}\rangle $\ is isomorphic to the free product $%
C=\ast \left\{ c_{i}^{\mathbb{Z}}:i\in \mathbb{N}\right\} $. By
construction, $C$ meets all conjugacy classes and is non-trivial. The set $%
S=\left\{ c_{i}\right\} $ is conjugation complete. Moreover, $\langle c_{i}:i\in \mathbb{N}\rangle
$ is a proper subgroup of $G$ because its abelianization is $C^{ab}=\prod_{i}\left(
\mathbb{Z}/\left\vert c_{i}\right\vert \right) ^{\mathbb{N}}$ (where $\left\vert c_{i}\right\vert
=0$ if $c_{i}$ is of infinite order ) \ is non-trivial whereas $G$ is a perfect group, even simple.
\end{proof}

\begin{theorem}
There are uncountably many isomorphism types of non-\ig~ groups of the form $%
\mathrm{PSL}_{m}\left( K\right) $, where $m\geq 2$ and $K$ is a countable field of infinite transcendence degree over $\mathbb{Q}$.
\end{theorem}

\begin{proof}
Let $P$ denote the set of all primes in $\mathbb{N}$ and let $R\subseteq P$ be any subset. Consider
the fields $\mathbb{Q}\left( \sqrt{R}\right) $ generated by square roots of all the primes in $R$.
Each of these fields is countable, and it's not too hard to see that they are pairwise
non-isomorphic. Now adjoin to each $\mathbb{Q}\left( \sqrt{R}\right) $ countably many
indeterminates $X=\left\{ X_{i}:i\in \mathbb{N}\right\}$,
i.e. consider the field of rational functions $\mathbb{Q}\left( \sqrt{R}%
,X\right) $ in indeterminates $X$. The fields $\mathbb{Q}\left( \sqrt{R}%
,X\right) $ are all countable, pairwise non-isomorphic and have infinite \ transcendence degree
over the prime field $\mathbb{Q}$.

\ \ By the result of Schreier and van der Waerden \cite{shreier-isom} the groups
$\mathrm{PSL}_{m}\left( K\right) ,$ $\mathrm{PSL}_{m^{\prime }}\left( K^{\prime }\right) ,$ $\left(
m,m^{\prime }\geq 2\right) $ can be isomorphic only when $m=m^{\prime }\,$, excluding the case of
$\mathrm{PSL}_{2}\left( \mathbb{F}_{7}\right) \simeq \mathrm{PSL}_{3}\left( \mathbb{F}_{2}\right)
$. Furthermore, if $m=m^{\prime }>2$, the isomorphism is possible only when the fields $K,K^{\prime
}$ are isomorphic. The same is true when $m=m^{\prime
}=2\,,$ excluding the case $\left\{ K,K^{\prime }\right\} =\left\{ \mathbb{F}%
_{4},\mathbb{F}_{5}\right\} $. We conclude that the groups $\mathrm{\mathrm{%
PSL}_{m}\left( \mathbb{Q}\left( \sqrt{R},X\right) \right) ,}m\geq 2,$ $%
R\subseteq P$, are pairwise non-isomorphic.

Similar construction can be carried over when the prime field is finite.
\end{proof}

\section{Open problems}
There are several open problems, concerning \ig.
\begin{enumerate}

\item : Is an infinite compact group \tig~ \iff it is connected abelian by profinite \cite{KLS2}?

\item : Is $\mathrm{SL}_{n}(\mathbb{Z}),(n\geq 3)$ \ig~ \cite{KLS2}?

\item : Is $\mathrm{SL}_{n}(\mathbb{Q}),(n\geq 2)$ \ig~  \cite{KLS2}?

\item : Is $\mathrm{SL}_{2}(\mathbb{R})$, viewed as an abstract group, \ig~ (Gelander)?

\item : Is every \ig~-linear group virtually solvable \cite{KLS2}?



\item  Is a finite index subgroup of an \ig~ group necessarily \ig~ (Wiegold)?
 According to the referee's advise this can be generalized as follows: Is a lattice in a \tig~ group an \ig~ group? Regretfully the answer is negative - as we proved $\sltr$ is \tig~ but contains a cocompact discrete group which is \ig~ by Gelander's result. It would be nice to describe Lie groups and their lattices for which the above question has a positive answer.

\item

``I think it is worth noting that every group outside \ig~ that I know satisfies no non-trivial
law, and it would be nice to have an example with a non-trivial law.
 For example, can a Tarski like group of finite exponent be outside \ig~?'' \cite{wiegold77}.
\end{enumerate}

\section{Acknowledgements}
\begin{acknowledgement}
 This paper was inspired by numerous discussions within the Midrasha seminar at
 the Weizmann's Institute of Science, which we admirably acknowledge.
 The research was supported by the Weizmann Institute of Science, SFB 701 of     Bielefeld University and
the program of fundamental scientific research of the SB RAS I.1.1., project 0314-2019-0004.

 Both authors are thankful to the referee for constructive comments and recommendations.
\end{acknowledgement}

\end{document}